\documentclass[11pt]{article}
\usepackage{amsfonts}
\usepackage{color}
\usepackage{amsmath,amssymb,comment}
\newtheorem{theo}{Theorem}[section]
\newtheorem{lem}[theo]{Lemma}
\newtheorem{cor}[theo]{Corollary}

\newtheorem{rem}[theo]{Remark}

\newcommand{\mysection}[1]{\section{#1} \setcounter{equation}{0}}
\newcommand{\proof}{{\sc Proof.} \quad}
\newcommand{\proofc}{{\sc Proof} \ }
\newcommand{\be}{\begin{equation} \label}
\newcommand{\ee}{\end{equation}}
\newcommand{\bea}{\begin{eqnarray}\label}
\newcommand{\eea}{\end{eqnarray}}
\newcommand{\bas}{\begin{eqnarray*}}
\newcommand{\eas}{\end{eqnarray*}}
\newcommand{\bit}{\begin{itemize}}
\newcommand{\eit}{\end{itemize}}
\newcommand{\qed}{\hfill$\Box$ \vskip.2cm}
\newcommand{\nn}{\nonumber}
\newcommand{\R}{\mathbb{R}}

\newcommand{\pO}{\partial\Omega}

\newcommand{\hra}{\hookrightarrow}
\newcommand{\io}{\int_\Omega}

\newcommand{\del}{\delta}
\newcommand{\al}{\alpha}
\newcommand{\lam}{\lambda}

\newcommand{\sig}{\sigma}

\newcommand{\bom}{\overline{\Omega}}
\newcommand{\Om}{\Omega}

\newcommand{\ov}{\overline}

\newcommand{\wh}{\widehat}

\newcommand{\hs}{\hspace*}

\newcommand{\vp}{\varphi}
\newcommand{\lbal}{\left\{ \begin{array}{l}}
\newcommand{\lball}{\left\{ \begin{array}{ll}}
\newcommand{\ear}{\end{array} \right.}

\newcommand{\abs}{\\[5pt]}

\newcommand{\adb}{\allowdisplaybreaks}
\newcommand{\tm}{T_{max}}

\begin{document}
\adb
\title{
Large-data global solutions to a quasilinear model for\\
viscuos acoustic wave propagation
in a non-isothermal setting}
\author{
Felix Meyer\footnote{felixmy@math.uni-paderborn.de}\\
{\small Universit\"at Paderborn, Institut f\"ur Mathematik}\\
{\small 33098 Paderborn, Germany}
\and
Michael Winkler\footnote{michael.winkler@math.uni-paderborn.de}\\
{\small Universit\"at Paderborn, Institut f\"ur Mathematik}\\
{\small 33098 Paderborn, Germany} }
\date{}
\maketitle
\begin{abstract}
\noindent 
The manuscript considers the model for conversion of mechanical energy into heat during acoustic wave propagation
in the presence of temperature-dependent elastic parameters, as given by
\bas
	\lbal
	u_{tt} = (\gamma(\Theta) u_{xt})_x + a (\gamma(\Theta) u_x)_x, \\[1mm]
	\Theta_t = D\Theta_{xx} + \gamma(\Theta) u_{xt}^2.
	\ear
	\qquad \qquad (\star)
\eas
It is firstly shown that when considered along with no-flux boundary conditions in an open bounded real interval $\Om$,
under the assumption that $\gamma\in C^2([0,\infty))$ is such that $\gamma>0$ and $\gamma'\ge 0$ on $[0,\infty)$ as well as
\bas
	D\cdot (\gamma+D) \cdot \gamma'' + 2\gamma \gamma'^2 \le 0
	\qquad \mbox{on } [0,\infty),
\eas
for all suitably regular initial data this problem admits a globally defined classical solution.
This complements recent findings in the literature, according to which ($\star$) may admit solutions blowing up in finite time
whenever $\gamma$ is positive and nondecreasing on $[0,\infty)$ with $\int_0^\infty \frac{d\xi}{\gamma(\xi)} < \infty$.\abs
Apart from that, it is found that if the additional assumption
\bas
	a|\Om|^2 \le \frac{\pi^2 \gamma(0)}{1+\sqrt{1+\frac{\gamma(0)}{D}}}
\eas
is satisfied, the all these solutions stabilize toward some spatially homogeneous equilibrium in the large time limit.\abs
\end{abstract}
%
%
%
%

%
%
%
%
%
%
\newpage
\section{Introduction}\label{intro}
This manuscript is concerned with the simplified model for the conversion of mechanical energy into heat during propagation of acoustic waves in elastic solids, as given by
\be{00}
	\lbal
	u_{tt} = (\gamma(\Theta) u_{xt})_x + a (\gamma(\Theta) u_x)_x, \\[1mm]
	\Theta_t = D\Theta_{xx} + \gamma(\Theta) u_{xt}^2,
	\ear
\ee
which arises as a simplified model for the conversion of mechanical energy into heat during propagation of acoustic waves
in elastic solids.
Motivated by experimental findings on corresponding material behavior in piezoceramics (\cite{friesen}),
in extension of classical wave equations of the form $u_{tt}=(a\gamma u_x)_x$ the system (\ref{00}) addresses situations in which
the elastic parameter $a\gamma$ may depend on the temperature $\Theta=\Theta(x,t)$ at the respective site. 
According to a core postulate in the modeling of Kelvin-Voigt type solids, (\ref{0}) additionally assumes frictional
generation of heat, with the corresponding viscosity exhibiting a temperature dependence parallel to the above
(\cite{friesen}, \cite{Gubinyi2007}).\abs
As seen in recent literature, sufficiently strong growth of the key ingredient $\gamma$ may enforce a finite-time collapse
of some solutions to (\ref{00}). In \cite{win_SIMA}, namely, it was found that if $a>0$ and $D>0$, and if 
$\gamma\in C^2([0,\infty))$ is a nondecreasing and positive function satisfying
\be{growth}
	\int_0^\infty \frac{d\xi}{\gamma(\xi)} < \infty,
\ee
then there exists a considerably large set of initial data for which (\ref{00}),
posed under homogeneous boundary conditions of Dirichlet type for $u$ and
of Neumann type for $\Theta$ in an open bounded interval $\Om$, admits a classical solution $(u,\Theta)$ in $\Om\times (0,T)$
with some $T>0$ at which blow-up occurs in the sense that
\be{hotspot}
	\limsup_{t\nearrow T} \|\Theta(\cdot,t)\|_{L^\infty(\Om)} = \infty.
\ee
In fact, the result in \cite{win_SIMA} actually applies to a model class more general than that in (\ref{00}), and 
straightforward adaptation of the arguments therein confirms that a similar conclusion can be drawn in the presence of
different boundary conditions, such as of homogeneous Neumann type for both components;
in this sense, phenomena of finite-time temperature hotspot formation in the style of (\ref{hotspot}) can
be viewed as a fairly robust consequence of the growth condition from (\ref{growth}) in systems of the form in (\ref{00}).
This is in sharp contrast to properties well-known as core characteristics of one-dimensional models for thermoelastic
evolution in the presense of constant elastic parameters:
Indeed, a broad literature addressing such scenarios 
has provided comprehensive results not only on global solvability
(\cite{racke_zheng}, \cite{kim}, \cite{jiang_QAM1993}, \cite{watson}, \cite{dafermos}, \cite{dafermos_hsiao_smooth},
\cite{slemrod}, \cite{bies_cieslak}),
but also on large time stabilization toward homogeneous states (\cite{racke_zheng}, \cite{hsiao_luo}, \cite{cieslak_MAAN}),
even in contexts both of classical thermoelasticity and of standard Kelvin-Voigt models
which are considerably more complex than (\ref{00}) by additionally accounting for thermal dilation meachanisms.
In suitably weakened form, some parts of these findings even extend to higher-dimensional systems involving
constant elastic coefficients (\cite{roubicek}, \cite{blanchard_guibe}, \cite{mielke_roubicek}, \cite{owczarek_wielgos},
\cite{cieslak_CVPD}).\abs
{\bf Main results.} \quad
The present manuscript aims at examining how far imposing suitable assumptions on temperature dependencies of parameters
may exclude phenomena of singularity formation.
In fact, for (\ref{00}) and close relatives the knowledge on large-time solvability to date seems rather limited:
For widely arbitrary $\gamma$ and initial data of arbitrary size, results on local existence can be derived 
within various concepts of solvability (\cite{fricke}, \cite{claes_win}, \cite{win_AMOP}), while in the presence
of temperature-dependent functions $\gamma$, global solutions so far seem to have been constructed only under appropriate
smallness conditions both on $a$ and the initial data (\cite{fricke}, \cite{claes_win}).
In \cite{meyer}, it has recently been shown that whenever suitably regular but possibly large initial data have been fixed, 
given $T>0$ one can identify a hypothesis on smallness of the derivative $\gamma'$ as sufficient to ensure 
the existence of a classical solution with lifespan exceeding $T$.
Except for the caveat from \cite{win_SIMA} concerned with functions $\gamma$ which grow in a superlinear manner by satisfying
(\ref{growth}), however, the literature does not provide substantial
information on the behavior of solutions emanating from large initial data.
In particular, it seems not even known whether bounded but nonconstant $\gamma$ may facilitate unboundedness phenomena
of the form in (\ref{hotspot}) or any other type; for such nonlinearities, only certain generalized solutions with possibly poor
regularity properties seem to have been found to exist globally so far (\cite{claes_lankeit_win}).\abs
In order to conveniently focus our considerations in this regard, unlike in several previous studies we shall here
concentrate on the apparently minimal model (\ref{00}) for thermoviscoelastic evolution involving temperature-dependent
parameters, hence neglecting any influence of thermal dilation effects.
Additionally including physically meaningful boundary conditions, we shall hence investigate the initial-boundary value problem
\be{0}
	\lball
	u_{tt} = (\gamma(\Theta) u_{xt})_x + a (\gamma(\Theta) u_x)_x,
	\qquad & x\in\Om, \ t>0, \\[1mm]
	\Theta_t = D\Theta_{xx} + \gamma(\Theta) u_{xt}^2,
	\qquad & x\in\Om, \ t>0, \\[1mm]
	u_x=0, \quad \Theta_x=0,
	\qquad & x\in\pO, \ t>0, \\[1mm]
	u(x,0)=u_0(x), \quad u_t(x,0)=u_{0t}(x), \quad \Theta(x,0)=\Theta_0(x),
	\qquad & x\in\Om,
	\ear
\ee
in an open bounded interval $\Om\subset\R$, with constant parameters $a>0$ and $D>0$, and with prescribed initial data
$u_0, u_{0t}$ and $\Theta_0$.\abs
An evident obstacle for the construction of large-data solutions to (\ref{0}) consists in the circumstance that, in general, 
the action of the nonlinear heat source $\gamma(\Theta) u_{xt}^2$ can apparently not be adequately be controlled on the basis
of elementary energy evolution properties of thermoelastic dynamics, as having formed powerful ingredients in essential
parts of existence theories in precedent related literature (\cite{racke_zheng}, \cite{bies_cieslak}).
Our approach to cope with this is based on the ambition to capture some essential aspects of the interaction in (\ref{0})
by tracing the time evolution of some spatial integrals which contain multiplicative couplings of both solution components.
Specifically, the quantities of central importance for our analysis will be functionals of the form
\be{en}
	y^{(B)}(t):= \io \frac{1}{\gamma(\Theta)+D} \cdot \Big(u_{xt}+au_x\Big)^2
	+ B \io u_x^2
\ee
with some appropriately chosen $B>0$.\abs
In fact, our first observation linked to these, to be detailed in Lemma \ref{lem2}, Lemma \ref{lem3} and Lemma \ref{lem4},
will utilize some suitably small parameters $B$ here to reveal some energy-like properties of the corresonding
functionals $y^{(B)}$ whenever
the smooth, positive and nondecreasing function $\gamma$, besides failing to fulfill (\ref{growth}), satisfies a structural
assumption slightly going beyond concavity (cf.~(\ref{g2})).
Using this information as a starting point for a suitable bootstrap procedure, as the first of our main results
we will obtain in Section \ref{sect2}
that, indeed, in the presence of any such $\gamma$ the problem (\ref{0}) admits a global classical solution:
\begin{theo}\label{theo9}
  Let $\Om\subset\R$ be an open bounded interval, let $a>0$ and $D>0$, and assume that
  \be{g1}
	\gamma\in C^2([0,\infty)) \quad \mbox{is such that \quad $\gamma>0$ and $\gamma'\ge 0$ on } [0,\infty),
  \ee
  and that
  \be{g2}
	D\cdot (\gamma+D) \cdot \gamma'' + 2\gamma \gamma'^2 \le 0
	\qquad \mbox{on } [0,\infty).
  \ee
  Then whenever 
  \be{init}
	\lbal
	u_0\in W^{3,2}(\Om) 
	\mbox{ satisfies $u_{0x}=0$ on $\pO$}, \\[1mm]
	u_{0t}\in W^{2,2}(\Om) 
	\mbox{ satisfies $(u_{0t})_x=0$ on $\pO$} \qquad \mbox{and} \\[1mm]
	\Theta_0\in W^{2,2}(\Om)
	\mbox{ satisfies $\Theta_0\ge 0$ in $\Om$ and $\Theta_{0x}=0$ on $\pO$,}
	\ear
  \ee
  the problem (\ref{0}) posseses a unique global classical solution $(u,\Theta)$ with
  \bas
	\lbal
	u\in \Big( \bigcup_{\mu\in (0,1)} C^{1+\mu,\frac{1+\mu}{2}}(\bom\times [0,\infty))\Big) 
		\cap C^{2,1}(\bom\times (0,\infty))
	\qquad \mbox{and} \\[1mm]
	0\le \Theta\in \Big( \bigcup_{\mu\in (0,1)} C^{1+\mu,\frac{1+\mu}{2}}(\bom\times [0,\infty))\Big) 
		\cap C^{2,1}(\bom\times (0,\infty))
	\ear
  \eas
  which is such that
  \bas
	\begin{array}{l}
	u_t\in 
	\Big( \bigcup_{\mu\in (0,1)} C^{1+\mu,\frac{1+\mu}{2}}(\bom\times [0,\infty))\Big) 
		\cap C^{2,1}(\bom\times (0,\infty)).
	\end{array}
  \eas
\end{theo}
Although (\ref{g2}) quite intricately links $\gamma$ to its first two derivatives, some instructive examples can readily
be constructed (cf.~the end of Section \ref{sect2}):
\begin{rem}\label{rem99}
  i) \ Monotonically stabilizing nonlinearities of the form
  \be{99.1}
	\gamma(\xi):=A-B e^{-\al \xi},
	\qquad \xi\ge 0,
  \ee
  can be seen to comply with the above assumptions whenever $\al>0$, $A>0$ and $B\in(0,A)$, provided that $D\ge\frac{2A}{1+\sqrt{8}}$.\abs
  ii) \ But also some unbounded $\gamma$ are admissible: 
  Namely, if $A>0$ and $B>0$, and if $D>0$ is large enough satisfying $D\ge 2B$, then Theorem \ref{theo9} applies to
  \be{99.2}
	\gamma(\xi):=A+B\ln (\xi+1),
	\qquad \xi\ge 0.
  \ee
\end{rem}
The potential of the functionals in (\ref{en}) to facilitate the discovery of essential properties of (\ref{0}), however,
appears to go substantially beyond the statement of Theorem \ref{theo9}.
If in addition to the above the number $a |\Om|^2$ is suitably small, namely, then on the basis of the Poincar\'e inequality
which states that
\be{P}
	\lam_1 \io \vp^2 \le \io \vp_x^2
	\quad \mbox{for all } \vp\in W_0^{1,2}(\Om),
	\quad \mbox{with} \quad
	\lam_1:=\frac{\pi^2}{|\Om|^2},
\ee
for some intermediate choice of $B$ more subtle
than the one underlying Theorem \ref{theo9}, the quantity accordingly defined in (\ref{en}) can actually be seen to
play the role of a genuinely decreasing energy functional (Lemma \ref{lem22}).
Appropriately exploiting this together with some quantitative information on corresponding dissipation rates, 
in the course of a second series of successively self-improving arguments we shall see in Section \ref{sect3} 
that in such constellations,
each of the solutions found above in fact stabilizes toward a semitrivial equilibroum in the large time limit:
\begin{theo}\label{theo33}
  Let $\Om\subset\R$ be an open bounded interval, let $a>0$ and $D>0$, and assume that (\ref{g1}) and (\ref{g2}) 
  are fulfilled with
  \be{aL}
	a|\Om|^2 \le \frac{\pi^2 \gamma(0)}{1+\sqrt{1+\frac{\gamma(0)}{D}}}.
  \ee
  Then whenever (\ref{init}) holds, one can find $\Theta_\infty\ge 0, \beta>0$ and $C>0$ such that the global classical solution $(u,\Theta)$
  of (\ref{0}) from Theorem \ref{theo9} satisfies
  \be{33.1}
	\bigg\| u(\cdot,t)-\frac{1}{|\Om|} \io (u_0+t u_{0t}) \bigg\|_{W^{1,\infty}(\Om)} \le C e^{-\beta t}
	\qquad \mbox{for all } t>0
  \ee
  and
  \be{33.2}
	\bigg\| u_t(\cdot,t)-\frac{1}{|\Om|} \io u_{0t} \bigg\|_{W^{1,\infty}(\Om)} \le C e^{-\beta t}
	\qquad \mbox{for all } t>0
  \ee
  as well as
  \be{33.3}
	\|\Theta(\cdot,t)-\Theta_\infty\|_{W^{1,\infty}(\Om)} \le C e^{-\beta t}
	\qquad \mbox{for all } t>0.
  \ee
\end{theo}
\mysection{Global solvabilty. Proof of Theorem \ref{theo9}}\label{sect2}
To guarantee the existence of classical solutions at least on local time intervals, let us state the following lemma 
which can be derived by adapting the proof of Theorem 1.1 in \cite{claes_win} in a straightforward manner.
\begin{lem}\label{lem_loc}
  Let $a>0$ and $D>0$, assume that $\gamma$ satisfies (\ref{g1}),
  and suppose that (\ref{init}) holds.
  Then there exist $\tm\in (0,\infty]$ as well as uniquely determined functions
  \be{l1}
	\lbal
	u\in \Big( \bigcup_{\mu\in (0,1)} C^{1+\mu,\frac{1+\mu}{2}}(\bom\times [0,\tm))\Big) \cap C^{2,1}(\bom\times (0,\tm))
		\qquad \mbox{and} \\[1mm]
	\Theta\in \Big( \bigcup_{\mu\in (0,1)} C^{1+\mu,\frac{1+\mu}{2}}(\bom\times [0,\tm))\Big) 
		\cap C^{2,1}(\bom\times (0,\tm))
	\ear
  \ee
  which are such that
  \be{l2}
	\begin{array}{l}
	u_t\in 
	\Big( \bigcup_{\mu\in (0,1)} C^{1+\mu,\frac{1+\mu}{2}}(\bom\times [0,\tm))\Big) \cap C^{2,1}(\bom\times (0,\tm)),
	\end{array}
  \ee
  that $\Theta\ge 0$ in $\Om\times (0,\tm)$, that $(u,\Theta)$ solves (\ref{0}) in the classical sense
  in $\Om\times (0,\tm)$, and which additionally are such that
  \bea{ext}
	\mbox{if $\tm<\infty$, \quad then \quad} 
	\limsup_{t\nearrow\tm} \Big\{ \|u_t(\cdot,t)\|_{W^{1,2}(\Om)} + \|\Theta(\cdot,t)\|_{L^\infty(\Om)}\Big\} = \infty.
  \eea
\end{lem}
Throughout the sequel, unless otherwise stated we shall tacitly assume that $a>0$ and $D>0$ and that (\ref{init}) holds.
Whenever $\gamma$ has been fixed in such a way that (\ref{g1}) holds, we then let $\tm$, $u$ and $\Theta$
be as obtained in Lemma \ref{lem_loc}, and following classical precedents in the analysis of related
problems (cf., e.g., \cite{kaltenbacher_lasiecka_pospieszalska}), we define
\be{v}
	v:=u_t+au,
\ee
noting that the triple $(v,u,\Theta)$ then forms a classical solution of
\be{0v}
	\lball
	v_t = (\gamma(\Theta) v_x)_x + av - a^2 u,
	\qquad & x\in\Om, \ t\in (0,\tm), \\[1mm]
	u_t = v-au,
	\qquad & x\in\Om, \ t\in (0,\tm), \\[1mm]
	\Theta_t = D\Theta_{xx} + \gamma(\Theta) u_{xt}^2,
	\qquad & x\in\Om, \ t\in (0,\tm), \\[1mm]
	v_x=0, \quad u_x=0, \quad \Theta_x=0,
	\qquad & x\in\pO, \ t\in (0,\tm), \\[1mm]
	v(x,0)=v_0(x), \quad u(x,0)=u_0(x), \quad \Theta(x,0)=\Theta_0(x),
	\qquad & x\in\Om,
	\ear
\ee
according to (\ref{0}).\abs
The cornerstone of our analysis can already be found in the following statement which combines both parabolic equations
in (\ref{0v}) and relies on (\ref{g2}) to a crucial extent.
\begin{lem}\label{lem2}
  Assume (\ref{g1}) and (\ref{g2}).
  Then whenever $\eta>0$,
  \be{2.1}
	\frac{d}{dt} \io \frac{1}{\gamma(\Theta)+D} v_x^2
	+ \frac{2\gamma(0)}{\gamma(0)+D} \io v_{xx}^2
	\le (2+\eta)a \io \frac{1}{\gamma(\Theta)+D} v_x^2
	+ \frac{a^3}{\eta D} \io u_x^2
  \ee
  for all $t\in(0,\tm).$
\end{lem}
\proof
  According to an integration by parts using (\ref{0v}),
  for all $t\in (0,\tm)$ we have
  \bea{2.2}
	\frac{d}{dt} \io \frac{1}{\gamma(\Theta)+D} v_x^2
	&=& 2 \io \frac{1}{\gamma(\Theta)+D} v_x v_{xt}
	- \io \frac{\gamma'(\Theta)}{(\gamma(\Theta)+D)^2} \Theta_t v_x^2 \nn\\
	&=& - 2 \io \Big( \frac{1}{\gamma(\Theta)+D} v_x\Big)_x \cdot \big( \gamma(\Theta) v_x\big)_x \nn\\
	& & + 2 \io \frac{1}{\gamma(\Theta)+D} v_x \cdot (av-a^2 u)_x \nn\\
	& & - D \io \frac{\gamma'(\Theta)}{(\gamma(\Theta)+D)^2} \Theta_{xx} v_x^2
	- \io \frac{\gamma'(\Theta) \gamma(\Theta)}{(\gamma(\Theta)+D)^2} u_{xt}^2 v_x^2,
  \eea
  where expanding yields
  \bea{2.3}
	& & \hs{-20mm}
	- 2 \io \Big( \frac{1}{\gamma(\Theta)+D} v_x\Big)_x \cdot \big( \gamma(\Theta) v_x\big)_x \nn\\
	&=& - 2 \io \Big( \frac{1}{\gamma(\Theta)+D} v_{xx} - \frac{\gamma'(\Theta)}{(\gamma(\Theta)+D)^2} \Theta_x v_x\Big) \cdot
		\big( \gamma(\Theta) v_{xx} + \gamma'(\Theta) \Theta_x v_x\big) \nn\\
	&=& - 2 \io \frac{\gamma(\Theta)}{\gamma(\Theta)+D} v_{xx}^2
	- 2 \io \frac{\gamma'(\Theta)}{\gamma(\Theta)+D} \Theta_x v_x v_{xx} \nn\\
	& & + 2 \io \frac{\gamma(\Theta)\gamma'(\Theta)}{(\gamma(\Theta)+D)^2} \Theta_x v_x v_{xx}
	+ 2 \io \frac{\gamma'^2(\Theta)}{(\gamma(\Theta)+D)^2} \Theta_x^2 v_x^2
	\qquad \mbox{for all } t\in (0,\tm).
  \eea
  Here, one further integration by parts shows that
  for all $t\in (0,\tm)$,
  \bas
	& & \hs{-20mm}
	- 2 \io \frac{\gamma'(\Theta)}{\gamma(\Theta)+D} \Theta_x v_x v_{xx} 
	+ 2 \io \frac{\gamma(\Theta)\gamma'(\Theta)}{(\gamma(\Theta)+D)^2} \Theta_x v_x v_{xx} \\
	&=& - \io \Big\{ \frac{\gamma'(\Theta)}{\gamma(\Theta)+D} - \frac{\gamma(\Theta)\gamma'(\Theta)}{(\gamma(\Theta)+D)^2} 
		\Big\} \cdot \Theta_x (v_x^2)_x \\
	&=& \io \Big\{ \frac{\gamma'(\Theta)}{\gamma(\Theta)+D} - \frac{\gamma(\Theta)\gamma'(\Theta)}{(\gamma(\Theta)+D)^2}
		\Big\} \cdot \Theta_{xx} v_x^2
	+ \io \Big( \frac{\gamma'}{\gamma+D} - \frac{\gamma \gamma'}{(\gamma+D)^2} \Big)'(\Theta) \cdot\Theta_x^2 v_x^2,
  \eas
  whence (\ref{2.3}) implies that
  \bea{2.4}
	& & \hs{-20mm}
	- 2 \io \Big( \frac{1}{\gamma(\Theta)+D} v_x\Big)_x \cdot \big( \gamma(\Theta) v_x\big)_x
	- D \io \frac{\gamma'(\Theta)}{(\gamma(\Theta)+D)^2} \Theta_{xx} v_x^2 \nn\\
	&=& - 2 \io \frac{\gamma(\Theta)}{\gamma(\Theta)+D} v_{xx}^2 \nn\\
	& & + 2 \io \Big\{ - D \frac{\gamma'(\Theta)}{(\gamma(\Theta)+D)^2} + \frac{\gamma'(\Theta)}{\gamma(\Theta)+D}
		- \frac{\gamma(\Theta)\gamma'(\Theta)}{(\gamma(\Theta)+D)^2} \Big\} \cdot \Theta_{xx} v_x^2 \nn\\
	& & + \io \bigg\{ 2\frac{\gamma'^2(\Theta)}{(\gamma(\Theta)+D)^2}
		+ \Big( \frac{\gamma'}{\gamma+D} - \frac{\gamma \gamma'}{(\gamma+D)^2}\Big)'(\Theta)\bigg\}
		\cdot \Theta_x^2 v_x^2
	\qquad \mbox{for all } t\in (0,\tm).
  \eea
  Since
  \bas
	& & \hs{-20mm}
	- D\frac{\gamma'(\xi)}{(\gamma(\xi)+D)^2} + \frac{\gamma'(\xi)}{\gamma(\xi)+D} 
		- \frac{\gamma(\xi)\gamma'(\xi)}{(\gamma(\xi)+D)^2} \\
	&=& \frac{-D\gamma'(\xi) + \gamma'(\xi) (\gamma(\xi)+D) - \gamma(\xi)\gamma'(\xi)}{(\gamma(\xi)+D)^2}
	= 0
	\qquad \mbox{for all } \xi>0,
  \eas
  since (\ref{g2}) ensures that
  \bas
	2\cdot\frac{\gamma'^2(\xi)}{(\gamma(\xi)+D)^2}
	+ \Big(\frac{\gamma'}{\gamma+D} - \frac{\gamma\gamma'}{(\gamma+D)^2}\Big)'(\xi)
	&=& 2\cdot\frac{\gamma'^2(\xi)}{(\gamma(\xi)+D)^2}
	+ \Big(\frac{D\gamma'}{(\gamma+D)^2}\Big)'(\xi) \\
	&=& \frac{2(\gamma(\xi)+D)\gamma'^2(\xi)}{(\gamma(\xi)+D)^3}
	+ \frac{D(\gamma(\xi)+D)\gamma''(\xi) - 2D\gamma'^2(\xi)}{(\gamma(\xi)+D)^3} \\
	&=& \frac{D(\gamma(\xi)+D) \gamma''(\xi) + 2\gamma(\xi)\gamma'^2(\xi)}{(\gamma(\xi)+D)^3} \\[2mm]
	&\le& 0
	\qquad \mbox{for all } \xi>0,
  \eas
  and since
  \bas
	- \io \frac{\gamma'(\Theta)\gamma(\Theta)}{(\gamma(\Theta)+D)^2} u_{xt}^2 v_x^2 \le 0
	\qquad \mbox{for all } t\in (0,\tm)
  \eas
  due to the nonnegativity of $\gamma'$ and $\gamma$, from (\ref{2.2}) and (\ref{2.4}) we thus infer that
  \bea{2.5}
	\hs{-8mm}
	\frac{d}{dt} \io \frac{1}{\gamma(\Theta)+D} v_x^2
	+ 2\io \frac{\gamma(\Theta)}{\gamma(\Theta)+D} v_{xx}^2
	&\le& 2 \io \frac{1}{\gamma(\Theta)+D} v_x \cdot (av-a^2u)_x \nn\\
	&=& 2a \io \frac{1}{\gamma(\Theta)+D} v_x^2
	- 2a^2 \io \frac{1}{\gamma(\Theta)+D} u_x v_x
  \eea
  for all $t\in (0,\tm)$.
  Using Young's inequality to confirm that
  \bas
	-2a^2 \io \frac{1}{\gamma(\Theta)+D} u_x v_x
	\le \eta a \io \frac{1}{\gamma(\Theta)+D} v_x^2
	+ \frac{a^3}{\eta} \io \frac{1}{\gamma(\Theta)+D} u_x^2
	\qquad \mbox{for all $t\in (0,\tm)$ and } \eta>0,
  \eas
  and simply estimating $\frac{1}{\gamma(\Theta)+D} \le \frac{1}{D}$ as well as 
  $\frac{\gamma(\Theta)}{\gamma(\Theta)+D} \ge \frac{\gamma(0)}{\gamma(0)+D}$, from (\ref{2.5}) we immediately obtain (\ref{2.1}).
\qed
The rightmost summand in (\ref{2.1}) can be controlled by using the ODE in (\ref{0v}) in a straightforward manner:
\begin{lem}\label{lem3}
  If (\ref{g1}) holds, and if $\eta>0$, then
  \be{3.1}
	\frac{d}{dt} \io u_x^2
	+ 2(1-\eta)a \io u_x^2 \le \frac{1}{2\eta a} \io v_x^2
	\qquad \mbox{for all } t\in (0,\tm).
  \ee
\end{lem}
\proof
  We only need to use the second equation in (\ref{0v}) along with Young's inequality to see that
  \bas
	\frac{1}{2} \frac{d}{dt} \io u_x^2
	= \io u_x\cdot (v-au)_x
	= - a\io u_x^2 + \io u_x v_x
	\qquad \mbox{for all } t\in (0,\tm),
  \eas
  and that here
  \bas
	\io u_x v_x \le \eta a \io u_x^2 
	+ \frac{1}{4\eta a} \io v_x^2
  \eas
  for all $t\in (0,\tm)$.
\qed
The above two lemmata can now be developed toward an energy-like inequality 
by using the Poincar\'e inequality (\ref{P}), here in a yet fairly rough manner in the sense that the optimal
constant therein does not play an essential role in the following argument.
\begin{lem}\label{lem4}
  Assume (\ref{g1}) and (\ref{g2}) . Then for each $T>0$ there exists $C(T)>0$ such that
  \be{4.1}
	\io \frac{1}{\gamma(\Theta(\cdot,t))+D} v_x^2(\cdot,t)
	\le C(T)
	\qquad \mbox{for all } t\in (0,T)\cap (0,\tm)
  \ee
  and
  \be{4.2}
	\io u_x^2(\cdot,t) 
	\le C(T)
	\qquad \mbox{for all } t\in (0,T)\cap (0,\tm)
  \ee
  as well as
  \be{4.3}
	\int_0^t \io v_{xx}^2 \le C(T)
	\qquad \mbox{for all } t\in (0,T)\cap (0,\tm).
  \ee
\end{lem}
\proof
  Taking $\lam_1>0$ as in the Poincar\'e inequality (\ref{P}), we fix $B>0$ such that with $c_1:=\frac{2\gamma(0)}{\gamma(0)+D}$
  we have
  \be{4.4}
	B \le \frac{c_1 a\lam_1}{2},
  \ee
  and let
  \be{4.5}
	y(t) := \io \frac{1}{\gamma(\Theta(\cdot,t))+D} v_x^2(\cdot,t)
	+ B \io u_x^2(\cdot,t),
	\qquad t\in [0,\tm).
  \ee
  Then a combination of (\ref{2.1}) with (\ref{3.1}) shows that
  \bea{4.6}
	\hs{-2mm}
	y'(t) + c_1 \io v_{xx}^2 + Ba \io u_x^2 
	\le 3a \io \frac{1}{\gamma(\Theta)+D} v_x^2
	+ \frac{a^3}{D} \io u_x^2 
	+ \frac{B}{a} \io v_x^2
	\quad \mbox{for all } t\in (0,\tm),
  \eea
  where by (\ref{P}) and (\ref{4.4}),
  \bas
	\frac{B}{a} \io v_x^2 
	\le \frac{B}{a\lam_1} \io v_{xx}^2
	\le \frac{c_1}{2} \io v_{xx}^2
	\qquad \mbox{for all } t\in (0,\tm).
  \eas
  Since furthermore
  \bas
	3a \io \frac{1}{\gamma(\Theta)+D} v_x^2
	+ \frac{a^3}{D} \io u_x^2
	\le \Big( 3a+\frac{a^3}{BD}\Big) \cdot y(t)
	\qquad \mbox{for all } t\in (0,\tm)
  \eas
  according to (\ref{4.5}), from (\ref{4.6}) we obtain that
  \be{4.7}
	y'(t) + \frac{c_1}{2} \io v_{xx}^2 \le c_2 y(t)
	\qquad \mbox{for all } t\in (0,\tm)
  \ee
  with $c_2:=3a+\frac{a^3}{BD}$.
  By Gr\"onwall's inequality, this implies that
  \bas
	y(t) \le c_3:=y(0) e^{c_2 T}
	\qquad \mbox{for all } t\in [0,T)\cap [0,\tm),
  \eas
  and a subsequent integration in (\ref{4.7}) leads to the inequality
  \bas
	\frac{c_1}{2} \int_0^t \io v_{xx}^2
	\le y(0) + c_2 \int_0^t y(s) ds
	\le c_3 + c_2 c_3 T
	\qquad \mbox{for all } t\in [0,T)\cap [0,\tm),
  \eas
  whence the proof can be completed by recalling the definition of $y$.
\qed
A simple embedding property develops the information on dissipation contained in (\ref{4.3}) into an estimate 
for first-order quantities:
\begin{lem}\label{lem5}
  Suppose that (\ref{g1}) and (\ref{g2}) are satisfied. Than for all $T>0$ there exists $C(T)>0$ such that
  \be{5.1}
	\int_0^t \|v_x(\cdot,s)\|_{L^\infty(\Om)}^2 ds \le C(T)
	\qquad \mbox{for all } t\in (0,T)\cap (0,\tm).
  \ee
\end{lem}
\proof
  Since by continuity of the embedding $W^{1,2}(\Om)\hra L^\infty(\Om)$ we can find $c_1>0$ such that
  \bas
	\|\vp\|_{L^\infty(\Om)} \le c_1\|\vp_x\|_{L^2(\Om)}
	\qquad \mbox{for all } \vp\in W_0^{1,2}(\Om),
  \eas
  it follows that
  \bas
	\int_0^t \|v_x(\cdot,s)\|_{L^\infty(\Om)}^2 ds
	\le c_1^2 \int_0^t \io v_{xx}^2
	\qquad \mbox{for all } t\in (0,\tm),
  \eas
  so that (\ref{5.1}) results from (\ref{4.2}).
\qed
This in turn ensures genuine $L^\infty$ boundedness of $u_x$ throughout finite time intervals:
\begin{lem}\label{lem6}
  Assuming (\ref{g1}) and (\ref{g2}), for each $T>0$ we can fix $C(T)>0$ in such a way that
  \be{6.1}
	|u_x(x,t)| \le C(T)
	\qquad \mbox{for all $x\in\Om$ and } t\in (0,T)\cap (0,\tm).
  \ee
\end{lem}
\proof
  In view of (\ref{0v}), we may represent $u_x$ according to
  \bas
	u_x(x,t)=e^{-at} u_{0x}(x) + \int_0^t e^{-a(t-s)} v_x(x,s) ds
  \eas
  for $x\in\Om$ and $t\in (0,\tm)$.
  By means of the Cauchy-Schwarz inequality, thanks to the nonnegativity of $a$ we can therefore estimate
  \bas
	\|u_x(\cdot,t)\|_{L^\infty(\Om)}
	&\le& e^{-at} \|u_{0x}\|_{L^\infty(\Om)} 
	+ \int_0^t e^{-a(t-s)} \|v_x(\cdot,s)\|_{L^\infty(\Om)} ds \\
	&\le& \|u_{0x}\|_{L^\infty(\Om)}
	+ \int_0^t \|v_x(\cdot,s)\|_{L^\infty(\Om)} ds \\
	&\le& \|u_{0x}\|_{L^\infty(\Om)}
	+ \bigg\{ \int_0^t \|v_x(\cdot,s)\|_{L^\infty(\Om)}^2 ds \bigg\}^\frac{1}{2} \cdot t^\frac{1}{2}
	\qquad \mbox{for all } t\in (0,\tm),
  \eas
  whereby (\ref{6.1}) becomes a consequence of Lemma \ref{lem5}.
\qed
As a combination of Lemma \ref{lem5} and Lemma \ref{lem6} suitably controls the heat source in (\ref{0v}),
a comparison argument can be applied so as to derive a pointwise upper bound also for the temperature variable.
\begin{lem}\label{lem7}
  Assume (\ref{g1}) and (\ref{g2}).
  Then for all $T>0$ there exists $C(T)>0$ such that
  \be{7.1}
	\Theta(x,t)\le C(T)
	\qquad \mbox{for all $x\in\Om$ and } t\in (0,T)\cap (0,\tm).
  \ee
\end{lem}
\proof
  From the third equation in (\ref{0v}) we obtain that thanks to (\ref{v}) and Young's inequality,
  \bas
	\Theta_t
	&=& D\Theta_{xx} + \gamma(\Theta) (v_x-au_x)^2 \\
	&\le& D\Theta_{xx} + \gamma(\Theta) \cdot (2v_x^2 + 2a^2 u_x^2)
	\qquad \mbox{in } \Om\times (0,\tm),
  \eas
  whence writing
  \bas
	h(t):=2\|v_x(\cdot,t)\|_{L^\infty(\Om)}^2
	+ 2a^2\|u_x(\cdot,t)\|_{L^\infty(\Om)}^2,
	\qquad t\in (0,\tm),
  \eas
  we see that
  \be{7.2}
	\Theta_t \le D\Theta_{xx} + h(t)\gamma(\Theta)
	\qquad \mbox{in } \Om\times (0,\tm),
  \ee
  where we note that as a consequence of Lemma \ref{lem5} and Lemma \ref{lem6}, for each $T>0$ one can find $c_1(T)>0$ satisfying
  \be{7.3}
	\int_0^t h(s) ds \le c_1(T)
	\qquad \mbox{for all } t\in (0,T)\cap (0,\tm).
  \ee
  We now use that $\gamma>0$ on $[0,\infty)$
  to see that letting $\xi_\star:=\|\Theta_0\|_{L^\infty(\Om)}$ and $\phi(\xi):=\int_{\xi_\star}^\xi \frac{d\sig}{\gamma(\sig)}$,
  $\xi\ge\xi_\star$, we obtain a strictly increasing function $\phi$ which
  maps $[\xi_\star,\infty)$ onto $[0,\infty)$, because (\ref{g2}) particularly ensures that $\gamma'' \le 0$ on $[0,\infty)$,
  meaning that $\gamma(\xi) \le \gamma(0) + \gamma'(0) \xi$ for all $\xi\ge 0$ and hence
  \bas
	\int_0^\infty \frac{d\xi}{\gamma(\xi)} = \infty.
  \eas
  Due to (\ref{7.3}) and the fact that $t\in (0,\tm)$ was arbitrary there, letting
  \bas
	z(t):=\phi^{-1}\bigg(\int_0^t h(s) ds\bigg),
	\qquad t\in [0,\tm),
  \eas
  we hence obtain a well-defined element $z$ of $C^1([0,\tm))$ which satisfies $\phi'(z(t))z'(t)=h(t)$ for all $t\in [0,\tm)$,
  and which consequently forms a solution of the initial-value problem
  \bas
	\lbal
	z'(t)=h(t)\gamma(z(t)),
	\qquad t\in (0,\tm), \\[1mm]
	z(0)=\xi_\star.
	\ear
  \eas
  Therefore, $\ov{\Theta}(x,t):=z(t)$, $(x,t)\in\bom\times [0,\tm)$, defines a function
  $\ov{\Theta}\in C^{2,1}(\bom\times [0,\tm))$ which satisfies $\ov{\Theta}_t = D\ov{\Theta}_{xx} + h(t)\gamma(\ov{\Theta})$
  in $\Om\times (0,\tm)$ with $\ov{\Theta}(x,0)=\xi_\star\ge\Theta(x,0)$ for all $x\in\Om$.
  By (\ref{7.2}) and a comparison principle, 
  we thus infer that $\Theta\le\ov{\Theta}$ in $\Om\times (0,\tm)$, and may conclude as intended.
\qed
Having this information at hand, we can now go back to (\ref{4.1}) to obtain an $L^2$ bound for $v_x$
which is independent of time within bounded intervals:
\begin{cor}\label{cor8}
  If (\ref{g1}) and (\ref{g2}) hold, then given any $T>0$, one can find $C(T)>0$ such that
  \be{8.1}
	\io v_x^2(\cdot,t) \le C(T)
	\qquad \mbox{for all } t\in (0,T)\cap (0,\tm).
  \ee
\end{cor}
\proof
  According to Lemma \ref{lem7}, for arbitrary $T>0$ we can fix $c_1(T)>0$ such that $\Theta\le c_1(T)$ in
  $\Om\times ((0,T)\cap (0,\tm))$. By nonnegativity of $\gamma'$, this means that
  \bas
	\frac{1}{\gamma(\Theta)+D} \ge \frac{1}{\gamma(c_1(T))+D}
	\qquad \mbox{in } \Om\times (0,T)\cap (0,\tm)),
  \eas
  so that (\ref{8.1}) is implied by (\ref{4.1}).
\qed
In view of Lemma \ref{lem_loc}, our main result on global solvability has thereby been achieved:\abs
\proofc of Theorem \ref{theo9}. \quad
  If the quantity $\tm$ from Lemma \ref{lem_loc} was finite, then Lemma \ref{lem4} and Corollary \ref{cor8} would provide
  $c_1>0$ and $c_2>0$ such that
  \be{9.99}
	\io u_x^2 \le c_1
	\quad \mbox{and} \quad
	\io v_x^2 \le c_2
	\qquad \mbox{for all } t\in (0,\tm),
  \ee
  while Lemma \ref{lem7} would yield $c_3>0$ fulfilling
  \be{9.999}
	\|\Theta\|_{L^\infty(\Om)} \le c_3
	\qquad \mbox{for all } t\in (0,\tm).
  \ee
  Once more using (\ref{v}) and Young's inequality, from (\ref{9.99}) we could infer that
  \bas
	\io u_{xt}^2
	\le 2\io v_x^2 + 2a^2 \io u_x^2
	\le 2c_2 + 2a^2 c_1
	\qquad \mbox{for all } t\in (0,\tm),
  \eas
  which together with (\ref{9.999}) would contradict (\ref{ext}).
\qed
Let us conclude this section by briefly verifying that our assumptions on $\gamma$ made in Theorem \ref{theo9}
indeed are satisfied by the simple-structured functions addressed in the introduction:\abs
\proofc of Remark \ref{rem99}. \quad
i) \ For 
$\gamma(\xi):=A-B e^{-\al \xi}$, $\xi\ge 0$,
with $\al>0$, $A>0$ and $B\in(0,A)$ such that $D\ge \frac{2A}{1+\sqrt{8}}$, we directly infer from $A-B<\gamma(\xi)<A$ for $\xi\ge0$ that \eqref{g1} holds. Furthermore, we may define 
\bas
	f(\xi):= D \cdot(\gamma(\xi)+D)\cdot \gamma''(\xi) +2\gamma(\xi) \gamma'(\xi)^2, \qquad\mbox{ for } \xi\ge0\nn
\eas
and conclude that $f$ attains its maximum over $[0,\infty)$ 
at the point $\xi_0=\max\left\{-\frac{1}{\al}\ln\left(\frac{2A+D}{4B}\right),0\right\}$ at which $f(\xi_0)$ 
is nonpositive as long as $\frac{7}{4}D^2+AD-A^2\ge 0$ holds. As this is ensured by our assumption $D\ge \frac{2A}{1+\sqrt{8}}$,
this function $\gamma$ actually also satisfies \eqref{g2}.\abs
ii) \ For $\gamma(\xi):= A+B\ln(\xi+1)$, $\xi\ge 0$,
with $A>0$ and $B>0$ fulfilling $D\ge 2B$, we directly see that \eqref{g1} holds, and computing
\bas
 	& & \hs{-20mm}
	D \cdot(\gamma(\xi)+D)\cdot \gamma''(\xi) +2\gamma(\xi) \gamma'(\xi)^2\nn\\
 	&=& D \cdot(A+D+B\ln(\xi+1))\cdot \frac{-B}{(\xi+1)^2} +2 (A+B\ln(\xi+1))\cdot \frac{B^2}{(\xi+1)^2}\nn\\
 	&=& \frac{B}{(\xi+1)^2}(2AB-AD-D^2)+\frac{B^2\ln(\xi+1)}{(\xi+1)^2}(2B-D) \qquad \mbox{ for }\xi\ge 0,
\eas
we conclude that our assumption on $B$ warrants that $\gamma$ satisfies \eqref{g2}. 
\qed
\mysection{Large time behavior. Proof of Theorem \ref{theo33}}\label{sect3}
Let us begin our description of the large time behavior in (\ref{0}) by observing that under the assumptions of Theorem \ref{theo33},
a combination of Lemma \ref{lem2} and Lemma \ref{lem3} slightly more subtle than in Lemma \ref{lem4} may even reveal some
temporal decay of the quantities addressed there:
\begin{lem}\label{lem22}
  Suppose that (\ref{g1}) and (\ref{g2}) hold, and that (\ref{aL}) is satisfied.
  Then there exist $C>0$ and $\beta>0$ such that
  \be{22.1}
	\io u_x^2(\cdot,t) \le C e^{-\beta t}
	\qquad \mbox{for all } t>0
  \ee
  and
  \be{22.2}
	\io \frac{1}{\gamma(\Theta(\cdot,t))+D} v_x^2(\cdot,t) \le C e^{-\beta t}
	\qquad \mbox{for all } t>0,
  \ee
  and furthermore we have
  \be{22.3}
	\int_0^\infty \io v_{xx}^2 < \infty.
  \ee
\end{lem}
\proof
  With $\lam_1$ as in (\ref{P}), we see that (\ref{aL}) is equivalent to the inequality
  \be{22.4}
	\frac{2\gamma(0)}{\gamma(0)+D} \cdot \frac{\lam_1}{a} > \frac{2}{\sqrt{D(\gamma(0)+D)}} + \frac{2}{\gamma(0)+D}.
  \ee
  Moreover, writing $\eta_1:=\sqrt{\frac{\gamma(0)+D}{D}}$ and $\eta_2:=\frac{1}{2}$, in a straightforward manner we obtain that
  \bas
	\inf_{\eta_1'>0, \eta_2'\in (0,1)} \Big\{ \frac{2+\eta_1'}{\gamma(0)+D} + \frac{1}{4\eta_1'\eta_2'(1-\eta_2')D}\Big\}
	&=& \frac{2+\eta_1}{\gamma(0)+D} + \frac{1}{4\eta_1 \eta_2 (1-\eta_2) D} \\
	&=& \frac{2+\sqrt{\frac{\gamma(0)+D}{D}}}{\gamma(0)+D}
		+ \frac{1}{\sqrt{\frac{\gamma(0)+D}{D}} \cdot D} \\
	&=& \frac{2}{\sqrt{D(\gamma(0)+D)}} + \frac{2}{\gamma(0)+D},
  \eas
  so that from (\ref{22.4}) it follows that
  \bas
	\frac{2\gamma(0)}{\gamma(0)+D} \cdot \frac{\lam_1}{a} - \frac{2+\eta_1}{\gamma(0)+D}
	> \frac{1}{4\eta_1 \eta_2(1-\eta_2)D}.
  \eas
  We can thus pick $B>0$ satisfying
  \bas
	\frac{2\gamma(0)}{\gamma(0)+D} \cdot \frac{\lam_1}{a} - \frac{2+\eta_1}{\gamma(0)+D} 
	> \frac{B}{2\eta_2 a^2} > \frac{1}{4\eta_1 \eta_2 (1-\eta_2)D},
  \eas
  which implies that the number
  \be{22.5}
	c_1:=2(1-\eta_2) aB - \frac{a^3}{\eta_1 D}
  \ee
  is positive, and that we can furthermore pick $\del\in (0,1)$ such that also
  \be{22.6}
	c_2:=(1-\del) \cdot \frac{2\gamma(0)\lam_1}{\gamma(0)+D} - \frac{(2+\eta_1)a}{\gamma(0)+D} - \frac{B}{2\eta_2 a}
  \ee
  is positive.
  Upon these choices, a combination of Lemma \ref{lem2} with Lemma \ref{lem3} shows that
  \bas
	y(t):=\io \frac{1}{\gamma(\Theta(\cdot,t))+D} v_x^2(\cdot,t) 
	+ B \io u_x^2(\cdot,t),
	\qquad t\ge 0,
  \eas
  has the property that
  \bas
	& & \hs{-20mm}
	y'(t) + \frac{2\gamma(0)}{\gamma(0)+D} \io v_{xx}^2
	+ 2(1-\eta_2) aB \io u_x^2 \\
	&\le& (2+\eta_1) a \io \frac{1}{\gamma(\Theta)+D} v_x^2
	+ \frac{a^3}{\eta_1 D} \io u_x^2
	+ \frac{B}{2\eta_2 a} \io v_x^2
	\qquad \mbox{for all } t>0,
  \eas
  where by the Poincar\'e inequality in (\ref{P}),
  \bas
	\frac{2\gamma(0)}{\gamma(0)+D} \io v_{xx}^2
	\ge \del\cdot \frac{2\gamma(0)}{\gamma(0)+D} \io v_{xx}^2
	+ (1-\del)\cdot \frac{2\gamma(0)}{\gamma(0)+D} \cdot \lam_1 \io v_x^2
	\qquad \mbox{for all } t>0,
  \eas
  because $\del<1$.
  Since 
  \be{22.7}
	\io \frac{1}{\gamma(\Theta)+D} v_x^2
	\le \frac{1}{\gamma(0)+D} \io v_x^2
	\qquad \mbox{for all } t>0
  \ee
  by monotonicity of $\gamma$, in line with (\ref{22.5}) and (\ref{22.6}) we therefore see that
  \bas
	y'(t) + \del\cdot\frac{2\gamma(0)}{\gamma(0)+D} \io v_{xx}^2
	+ c_1 \io u_x^2 + c_2 \io v_x^2 \le 0
	\qquad \mbox{for all } t>0,
  \eas
  whence again using (\ref{22.7}) we infer that
  \be{22.8}
	y'(t) + c_3 y(t) + \del\cdot\frac{2\gamma(0)}{\gamma(0)+D} \io v_{xx}^2 \le 0
	\qquad \mbox{for all } t>0
  \ee
  with $c_3:=\min\{ \frac{c_1}{B} \, , \, (\gamma(0)+D) c_2\}$.
  An ODE comparison argument shows that therefore
  \bas
	y(t) \le y(0) e^{-c_3 t}
	\qquad \mbox{for all } t>0,
  \eas
  while a simple integration in (\ref{22.8}) furthermore yields the inequality
  \bas
	\del\cdot\frac{2\gamma(0)}{\gamma(0)+D} \int_0^t \io v_{xx}^2 \le y(0)
  \eas
  for all $t>0$.
\qed
By interpolation, (\ref{22.1}) and (\ref{22.3}) imply exponential decay of $u_x$ also with respect to the norm
in $L^\infty(\Om)$.
\begin{lem}\label{lem23}
  Assume (\ref{g1}) and (\ref{g2}), and suppose that (\ref{aL}) holds.
  Then there exist $C>0$ and $\beta>0$ such that
  \be{23.1}
	\|u_x(\cdot,t)\|_{L^\infty(\Om)} \le C e^{-\beta t}
	\qquad \mbox{for all } t>0.
  \ee
\end{lem}
\proof
  A rough estimation using the second equation in (\ref{0v}) and the Cauchy-Schwarz inequality shows that
  \bas
	\|u_{xx}\|_{L^2(\Om)}
	&=& \bigg\| e^{-at} u_{0xx} + \int_0^t e^{-a(t-s)} v_{xx}(\cdot,s) ds \bigg\|_{L^2(\Om)} \\
	&\le& \|u_{0xx}\|_{L^2(\Om)} + \int_0^t \|v_{xx}(\cdot,s)\|_{L^2(\Om)} ds \\
	&\le& \|u_{0xx}\|_{L^2(\Om)} + \bigg\{ \int_0^t \io v_{xx}^2 \bigg\}^\frac{1}{2} \cdot t^\frac{1}{2}
	\qquad \mbox{for all } t>0,
  \eas
  whence relying on (\ref{22.3}) we can fix $c_1>0$ such that
  \bas
	\|u_{xx}\|_{L^2(\Om)} \le c_1 \cdot (1+t^\frac{1}{2})
	\qquad \mbox{for all } t>0.
  \eas
  Since, on the other hand, (\ref{22.1}) provides $c_2>0$ and $\beta_1>0$ fulfilling
  \bas
	\|u_x\|_{L^2(\Om)} \le c_2 e^{-\beta_1 t}
	\qquad \mbox{for all } t>0,
  \eas
  and since a Gagliardo-Nirenberg inequality provides $c_3>0$ such that
  \bas
	\|\vp\|_{L^\infty(\Om)} \le c_3 \|\vp_x\|_{L^2(\Om)}^\frac{1}{2} \|\vp\|_{L^2(\Om)}^\frac{1}{2}
	\qquad \mbox{for all } \vp\in W_0^{1,2}(\Om),
  \eas
  this implies that
  \bas
	\|u_x\|_{L^\infty(\Om)} \le c_3 \cdot (c_2 e^{-\beta_1 t})^\frac{1}{2} \cdot 
		\big\{ c_1\cdot (1+t^\frac{1}{2})\big\}^\frac{1}{2}
	\qquad \mbox{for all } t>0,
  \eas
  and thereby entails (\ref{23.1}) with any fixed $\beta\in (0,\frac{\beta_1}{2})$, and with
  $C:=\sqrt{c_1 c_2} c_3\cdot\sup_{t>0} \big\{ \sqrt{1+t^\frac{1}{2}} e^{-(\frac{\beta_1}{2}-\beta)t}\big\}$ clearly being finite.
\qed
In Lemma \ref{lem25}, the previous statement will be combined with the following.
\begin{lem}\label{lem24}
  If (\ref{g1}), (\ref{g2}) and (\ref{aL}) hold, then
  \be{24.1}
	\int_0^\infty \|v_x(\cdot,t)\|_{L^\infty(\Om)}^2 dt < \infty.
  \ee
\end{lem}
\proof
  This can be seen by using (\ref{22.3}) together with the continuity of the embedding $W_0^{1,2}(\Om) \hra L^\infty(\Om)$,
  quite in the style of the argument in Lemma \ref{lem5}.
\qed
In fact, we can now control the growth in the temperture distribution as follows.
\begin{lem}\label{lem25}
  Assuming (\ref{g1}), (\ref{g2}) and (\ref{aL}), one can find $C>0$ such that
  \be{25.1}
	\Theta(x,t) \le C
	\qquad \mbox{for all $x\in\Om$ and } t>0.
  \ee
\end{lem}
\proof
  Since a combination of Lemma \ref{lem24} with Lemma \ref{lem23} clearly shows that
  \bas
	\int_0^\infty \big( \|v_x(\cdot,t)\|_{L^\infty(\Om)}^2 + \|u_x(\cdot,t)\|_{L^\infty(\Om)}^2 \big) dt <\infty,
  \eas
  we can derive (\ref{25.1}) by an almost verbatim copy of the comparison argument detailed in Lemma \ref{lem7}.
\qed
With this information at hand, we can now favorably control the temperature-dependent weight function appearing in (\ref{22.2}),
so as to obtain a genuine exponential decay property also of $v_x$ in the following sense.
\begin{cor}\label{cor255}
  If (\ref{g1}), (\ref{g2}) and (\ref{aL}) are valid, then with some $C>0$ and $\beta>0$ we have
  \bas
	\io v_x^2(\cdot,t) \le C e^{-\beta t}
	\qquad \mbox{for all } t>0.
  \eas
\end{cor}
\proof
  We only need to note that the boundedness property of $\Theta$ asserted by Lemma \ref{lem25} ensures that
  $\inf_{(x,t)\in \Om\times (0,\infty)} \frac{1}{\gamma(\Theta(x,t))+D}$ is positive according to (\ref{g1}).
  The claim therefore is is an immediate consequence of (\ref{22.2}).
\qed
Due to parabolic smoothing, the boundedness features particularly entailed by
Lemma \ref{lem25}, Lemma \ref{lem24} and Lemma \ref{lem22} imply some time-independent H\"older regularity property
of $\Theta$.
\begin{lem}\label{lem26}
  If (\ref{g1}), (\ref{g2}) and (\ref{aL}) hold, then there exist $\alpha\in (0,1)$ and $C>0$ such that
  \be{26.1}
	\|\Theta\|_{C^{\al,\frac{\al}{2}}(\bom\times [t_0,t_0+1])} \le C
	\qquad \mbox{for all } t_0>1.
  \ee
\end{lem}
\proof
  Since $\Theta$ and hence also $\gamma(\Theta)$ are bounded in $\Om\times (0,\infty)$ by Lemma \ref{lem25}, from Lemma \ref{lem24},
  Lemma \ref{lem22} and (\ref{v}) it follows that there exists $c_1>0$ fulfilling
  \be{26.2}
	\int_{t_0-1}^{t_0+1} \big\| \gamma(\Theta(\cdot,t)) u_{xt}^2(\cdot,t)\big\|_{L^\infty(\Om)}^2 dt
	+ \int_{t_0-1}^{t_0+1} \big\| \Theta(\cdot,t)\big\|_{L^\infty(\Om)}^2 dt \le c_1
	\qquad \mbox{for all } t_0>1.
  \ee
  Employing a known result on H\"older regularity in scalar parabolic equations (\cite[Theorem 1.3 and Remark 1.4]{PV}), 
  we can thereupon pick
  $\al\in (0,1)$ and $c_2>0$ in such a way that whenever $t_0>1, h\in C^0(\bom\times [t_0-1,t_0+1])$ and
  $w\in C^{2,1}(\bom\times [t_0-1,t_0+1])$ are such that
  \be{26.3}
	\lball
	w_t = D w_{xx} + h(x,t), 
	\qquad & x\in\Om, \ t\in (t_0-1,t_0+1), \\[1mm]
	w_x=0,
	\qquad & x\in\pO, \ t\in (t_0-1,t_0+1), \\[1mm]
	w(x,t_0-1)=0,
	\qquad & x\in\Om,
	\ear
  \ee
  is solved in the classical sense, with 
  \be{26.4}
	\int_{t_0-1}^{t_0+1} \|h(\cdot,t)\|_{L^\infty(\Om)}^2 dt \le 10 c_1,
  \ee
  we have
  \be{26.5}
	\|w\|_{C^{\al,\frac{\al}{2}}(\bom\times [t_0-1,t_0+1])} \le c_2.
  \ee
  To appropriately make use of this, given $t_0>1$ we fix $\zeta\in C^\infty([t_0-1,t_0+1])$ such that
  $\zeta(t_0-1)=0$, $\zeta\equiv 1$ on $[t_0,t_0+1]$, $0\le\zeta\le 1$ and $|\zeta'|\le 2$,
  and let $w(x,t):=\zeta(t) \Theta(x,t)$ as well as
  $h(x,t):=\zeta(t) \gamma(\Theta(x,t))u_{xt}^2(x,t) + \zeta'(t) \Theta(x,t)$
  for $(x,t)\in\bom\times [t_0-1,t_0+1]$.
  Then from Lemma \ref{lem_loc} and (\ref{0}) we obtain that $h$ and $w$ have the regularity properties listed above, and that 
  (\ref{26.3}) is classically solved, whereas (\ref{26.2}) along with Young's inequality ensures that
  \bas
	\int_{t_0-1}^{t_0+1} \|h(\cdot,t)\|_{L^\infty(\Om)}^2 dt
	&\le& 2 \int_{t_0-1}^{t_0+1} \big\| \zeta(t) \gamma(\Theta(\cdot,t)) u_{xt}^2(\cdot,t)\big\|_{L^\infty(\Om)}^2 dt
	+ 2 \int_{t_0-1}^{t_0+1} \big\| \zeta'(t) \Theta(\cdot,t)\big\|_{L^\infty(\Om)}^2 dt \\
	&\le& 2c_1 + 8c_1 = 10c_1.
  \eas
  As thus also (\ref{26.4}) holds, the conclusion (\ref{26.5}) thereby implied especially entails (\ref{26.1}) with $C:=c_2$,
  because $\zeta\equiv 1$ on $[t_0,t_0+1]$.
\qed
In order to similarly derive a H\"older bound also for the first component $v$ of the solution to (\ref{0v}), 
let us first record a basic evolution property of two mass functionals related to (\ref{0}).
\begin{lem}\label{lem_mass}
  If (\ref{g1}) holds, then
  \be{mass}
	\io u_t(\cdot,t)= \io u_{0t}
	\quad \mbox{and} \quad
	\io u(\cdot,t) = \io u_0 + t\io u_{0t}
	\qquad \mbox{for all } t\in (0,\tm).
  \ee
\end{lem}
\proof
  Both identities immediately result from the fact that according to (\ref{0}) we have $\frac{d^2}{dt^2} \io u=0$ 
  for all $t\in (0,\tm)$.
\qed
Upon simple linear combination of the components $u$ and $v$ with certain spatially flat functions,
the first equation in (\ref{0v}) can indeed be made accessible to standard parabolic H\"older regularity theory
in a reasonably time-independent manner:
\begin{lem}\label{lem27}
  Suppose that (\ref{g1}) and (\ref{g2}) and (\ref{aL}) are satisfied.
  Then there exist $\alpha\in (0,1)$ and $C>0$ such that
  \be{27.1}
	\|v_x\|_{C^{\al,\frac{\al}{2}}(\bom\times [t_0,t_0+1])} \le C
	\qquad \mbox{for all } t_0>2.
  \ee
\end{lem}
\proof
  We let 
  \bas
	\wh{u}(x,t):=u(x,t)-\frac{1}{|\Om|} \cdot \bigg\{ \io u_0 + t \io u_{0t} \bigg\}
  \eas
  as well as
  \bas
	\wh{v}(x,t):=v(x,t)-\frac{1}{|\Om|} \cdot \bigg\{ a \io u_0 + (1+at) \io u_{0t} \bigg\}
  \eas
  for $x\in\Om$ and $t>0$. Then drawing on (\ref{mass}) we readily confirm that
  \be{mass0}
	\io \wh{u} = \io \wh{v} = 0
	\qquad \mbox{for all } t>0,
  \ee
  and that according to (\ref{0v}),
  \be{27.2}
	\wh{v}_t = \big(A(x,t) \wh{v}_x\big)_x + h(x,t),
	\qquad x\in\Om, \ t>0,
  \ee
  where $A(x,t):=\gamma(\Theta(x,t))$ and $h(x,t):=a\wh{v}(x,t)-a^2 \wh{u}(x,t)$, $(x,t)\in\Om\times (0,\infty)$.
  Here, Lemma \ref{lem25}, Lemma \ref{lem26} and (\ref{g1}) provide positive constants $c_1, c_2, c_3$ and $\al_1$ such that
  \be{27.3}
	c_1 \le A \le c_2
	\quad \mbox{in } \Om\times (0,\infty)
	\qquad \mbox{and} \qquad
	\|A(\cdot,t)\|_{C^{\al_1}(\bom)} \le c_3
	\quad \mbox{for all } t>1,
  \ee
  while a combination of Lemma \ref{lem22} with Corollary \ref{cor255} shows that thanks to 
  (\ref{mass0})
  and the continuity of the embedding $W^{1,2}(\Om) \hra L^\infty(\Om)$ we can find $c_4>0$ such that
  \be{27.4}
	|h(x,t)| \le c_4
	\qquad \mbox{for all $x\in\Om$ and } t>0.
  \ee
  An argument based on localization in time, quite in the flavor of that in Lemma \ref{lem26}, now enables us to apply
  a standard result on parabolic gradient regularity (\cite[Theorem 1.1]{lieberman}) 
  to infer from (\ref{27.2}), (\ref{27.3}) and (\ref{27.4})
  that, indeed, (\ref{27.1}) holds with some suitably small $\al\in (0,1)$ and some adequately large $C>0$.
\qed
In order to prepare a suitable interpolation between the estimates from Lemma \ref{lem27} and Corollary \ref{cor255},
but also a similar argument addressing the temperture variable in Lemma \ref{lem31} below,
we include a brief proof of a Gagliardo-Nirenberg type inequality which, up to explicit information of constants involved,
is essentially well-known:
\begin{lem}\label{lem21}
  Let $\al\in (0,1)$ and $p\in [1,\infty)$. Then for each $\vp\in C^\al(\bom)$,
  \be{21.1}
	\|\vp\|_{L^\infty(\Om)} 
	\le 
	\frac{p\al+1}{(p\al)^\frac{p\al}{p\al+1}}
		\cdot [\vp]_{\al;\bom}^\frac{1}{p\al+1} \cdot \|\vp\|_{L^p(\Om)}^\frac{p\al}{p\al+1}
	\ + \ \frac{p\al+1}{p\al |\Om|^\frac{1}{p}} \cdot \|\vp\|_{L^p(\Om)},
  \ee
  where we have set
  \bas
	[\psi]_{\al;\bom} := \sup_{\begin{array}{c} \\[-6mm] \scriptstyle x\in\bom, y\in\bom \\[-2mm] \scriptstyle x\ne y \end{array}} \frac{|\psi(x)-\psi(y)|}{|x-y|^\al}
  \eas
  for $\psi\in C^\al(\bom)$.
\end{lem}
\proof
  We let
  \be{21.2}
	d:=(p\al)^{-\frac{p}{p\al+1}},
  \ee
  and for $\vp\in C^\al(\bom)$, we moreover abbreviate $I:=\|\vp\|_{L^p(\Om)}$ and
  $H:=[\vp]_{\al;\bom}$, noting that $H\ne 0$, and that we may thus set 
  $\del:=\min\{ d\cdot (\frac{I}{H})^\kappa \, , \, |\Om|\}$ with $\kappa:=\frac{p}{p\al+1}$.
  Then since $\del\le |\Om|$, for each fixed $x\in\bom$ we can find a subinterval $\Om_0=\Om_0(x)\subset\Om$ such that
  $|\Om_0|=\del$ and $x\in \bom_0$,
  and thanks to the continuity of $\vp$ we can furthermore pick $x_0=x_0(x)\in\bom_0$ with the property that
  \bas
	I^p \ge \int_{\Om_0} |\vp(y)|^p dy = \del |\vp(x_0)|^p,
  \eas
  so that $|\vp(x_0)| \le \del^{-\frac{1}{p}} I$.
  Since $|\vp(x)-\vp(x_0)| \le |x-x_0|^\al \cdot H \le \del^\al H$ due to the fact that $x_0\in\bom_0$, we thus obtain that
  \be{21.4}
	|\vp(x)|
	\le |\vp(x_0)| + |\vp(x)-\vp(x_0)| 
	\le \del^{-\frac{1}{p}} I + \del^\al H,
  \ee
  which in the case when $|\Om| \ge d\cdot (\frac{I}{H})^\kappa$ implies that
  \bea{21.5}
	|\vp(x)|
	&\le& d^{-\frac{1}{p}} \cdot \Big(\frac{I}{H}\Big)^{-\frac{\kappa}{p}} I + d^\al \cdot \Big(\frac{I}{H}\Big)^{\al\kappa} H 
	\nn\\
	&=& (d^{-\frac{1}{p}} + d^\al) H^\frac{1}{p\al+1} I^\frac{p\al}{p\al+1},
  \eea
  because $\frac{\kappa}{p}=1-\al\kappa=\frac{1}{p\al+1}$ and $1-\frac{\kappa}{p}=\al\kappa=\frac{p\al}{p\al+1}$
  according to the choice of $\kappa$.
  As $d^{-\frac{1}{p}}+d^\al=(p\al)^\frac{1}{p\al+1} + (p\al)^{-\frac{p\al}{p\al+1}} = (p\al)^{-\frac{p\al}{p\al+1}} \cdot (p\al+1)$
  by (\ref{21.2}), we thus obtain (\ref{21.1}) from (\ref{21.5}) in this case.\abs
  If, conversely, $|\Om| < d\cdot (\frac{I}{H})^\kappa$ and hence $\del=|\Om|$, 
  then $H<d^\frac{1}{\kappa} |\Om|^{-\frac{1}{\kappa}} I$, so that (\ref{21.4}) ensures that
  \bas
	|\vp(x)| 
	\le \Big\{ |\Om|^{-\frac{1}{p}} + |\Om|^\al \cdot d^\frac{1}{\kappa} |\Om|^{-\frac{1}{\kappa}} \Big\} \cdot I,
  \eas
  and thereby entails (\ref{21.1}) also in any such situation, as $\al-\frac{1}{\kappa}=-\frac{1}{p}$
  and $1+d^\frac{1}{\kappa}=1+(p\al)^{-1}=\frac{p\al+1}{p\al}$.
\qed
{\bf Remark.} \quad
  By elementary minimization, it can be seen that the number $d$ in (\ref{21.2}) satisfies
  \bas
	d^{-\frac{1}{p}} + d^\al = \min_{\xi>0} \big(\xi^{-\frac{1}{p}}+\xi^\al\big).
  \eas
  Therefore, (\ref{21.1}) appears as the best consequence that can be drawn from (\ref{21.4}).\abs%
Indeed, we can thereby improve decay statement made in Corollary \ref{cor255}
with respect to its topological part.
\begin{lem}\label{lem28}
  Assume (\ref{g1}) and (\ref{g2}) as well as (\ref{aL}).
  Then there exist $\beta>0$ and $C>0$ such that
  \be{28.1}
	\|v_x(\cdot,t)\|_{L^\infty(\Om)} \le C e^{-\beta t}
	\qquad \mbox{for all } t>2.
  \ee
\end{lem}
\proof
  We first invoke Lemma \ref{lem27} to pick $\al>0$ and $c_1>0$ such that
  \be{28.2}
	\|v_x(\cdot,t)\|_{C^\al(\bom)} \le c_1
	\qquad \mbox{for all } t>2,
  \ee
  while Corollary \ref{cor255} provides $\beta_1>0$ and $c_2>0$ fulfilling
  \be{28.3}
	\|v_x(\cdot,t)\|_{L^2(\Om)} \le c_2 e^{-\beta_1 t}
	\qquad \mbox{for all } t>0.
  \ee
  We then rely on Lemma \ref{lem21} in choosing $c_3>0$ in such a way that
  \bas
	\|\vp\|_{L^\infty(\Om)}
	\le c_3 \|\vp\|_{C^\al(\bom)}^\frac{1}{2\al+1} \|\vp\|_{L^2(\Om)}^\frac{2\al}{2\al+1}
	+ c_3 \|\vp\|_{L^2(\Om)}
	\qquad \mbox{for all $\vp\in C^\al(\bom)$},
  \eas
  and thereby conclude from (\ref{28.2}) and (\ref{28.3}) that
  \bas
	\|v_x\|_{L^\infty(\Om)} \le c_3 \cdot c_1^\frac{1}{2\al+1} \cdot \big( c_2 e^{-\beta_1 t}\big)^\frac{2\al}{2\al+1}
	+ c_3 c_2 e^{-\beta_1 t}
	\qquad \mbox{for all } t>2,
  \eas
  which entails (\ref{28.1}) with $C:=c_1^\frac{1}{2\al+1} c_2^\frac{2\al}{2\al+1} c_3 + c_2 c_3$ and
  $\beta:=\frac{2\al\beta_1}{2\al+1}$.
\qed
Together with Lemma \ref{lem25} and Lemma \ref{lem23}, this especially entails eventual
boundedness of the heat source in (\ref{0v}),
so that again standard parabolic theory can be applied to deduce H\"older bounds also for $\Theta_x$.
\begin{lem}\label{lem29}
  If (\ref{g1}), (\ref{g2}) and (\ref{aL}) are satisfied, then there exist $\al>0$ and $C>0$ such that
  \be{29.1}
	\|\Theta_x\|_{C^{\al,\frac{\al}{2}}(\bom\times [t_0,t_0+1])} \le C
	\qquad \mbox{for all } t_0>3.
  \ee
\end{lem}
\proof
  The second equation in (\ref{0}) can be recast in the form
  \be{29.2}
	\Theta_t = D\Theta_{xx} + h(x,t),
	\qquad (x,t)\in\Om\times (0,\infty),
  \ee
  where Lemma \ref{lem25}, Lemma \ref{lem28} and Lemma \ref{lem23} particularly ensure that 
  $h(x,t):=\gamma(\Theta(x,t)) (v_x(x,t)-au_x(x,t))^2$, $(x,t)\in\Om\times (0,\infty)$, 
  defines a function $h\in L^\infty(\Om\times (2,\infty))$.
  Again relying on a suitable cut-off procedure similar to that from Lemma \ref{lem26}, we may thus conclude (\ref{29.1}) from
  (\ref{29.2}) by one more application of the gradient H\"older estimates established in \cite[Theorem 1.1]{lieberman}.
\qed
Again, this boundedness property can be supplemented by some decay information with respect to a coarser topology:
\begin{lem}\label{lem30}
  Let (\ref{g1}), (\ref{g2}) and (\ref{aL}) hold. Then, with some $\beta>0$ and $C>0$, we have
  \be{30.1}
	\io \Theta_x^2(\cdot,t) \le C e^{-\beta t}
	\qquad \mbox{for all } t>3.
  \ee
\end{lem}
\proof
  A straightforward testing procedure applied to the second equation in (\ref{0}) shows that thanks to (\ref{v}) and
  Young's inequality,
  \bas
	\frac{1}{2} \frac{d}{dt} \io \Theta_x^2 + D \io \Theta_{xx}^2
	&=& - \io \gamma(\Theta) (v_x-au_x)^2 \Theta_{xx} \\
	&\le& \frac{D}{2} \io \Theta_{xx}^2
	+ \frac{1}{2D} \io \gamma^2(\Theta) (v_x-au_x)^4 \\
	&\le& \frac{D}{2} \io \Theta_{xx}^2
	+ \frac{4}{D} \io \gamma^2(\Theta) v_x^4
	+ \frac{4a^4}{D} \io \gamma^2(\Theta) u_x^4
	\qquad \mbox{for all } t>0.
  \eas
  Again using that $\gamma(\Theta)$ is bounded in $\Om\times (0,\infty)$ by Lemma \ref{lem25} and (\ref{g1}),
  we thus infer the existence of $c_1>0$ such that
  \bas
	\frac{d}{dt} \io \Theta_x^2 + D \io \Theta_{xx}^2
	\le c_1 \io v_x^4 + c_1 \io u_x^4
	\qquad \mbox{for all } t>0,
  \eas
  where according to the Poincar\'e inequality (\ref{P}),
  \bas
	D \io \Theta_{xx}^2 \ge D\lam_1 \io \Theta_x^2
	\qquad \mbox{for all } t>0,
  \eas
  and where Lemma \ref{lem23} along with Lemma \ref{lem28} yields $\beta_1\in (0,D\lam_1)$ and $c_2>0$ such that
  \bas
	c_1 \io v_x^4 + c_1 \io u_x^4
	\le c_2 e^{-\beta_1 t}
	\qquad \mbox{for all } t>3.
  \eas
  Therefore,
  \bas
	\frac{d}{dt} \io \Theta_x^2 + D\lam_1 \io \Theta_x^2
	\le c_2 e^{-\beta_1 t}
	\qquad \mbox{for all } t>3,
  \eas
  so that
  \bas
	\io \Theta_x^2(\cdot,t)
	\le e^{-\lam_1(t-3)} \io \Theta_x^2(\cdot,3)
	+ c_2 \int_3^t e^{-D\lam_1(t-s)} e^{-\beta_1 s} ds
	\qquad \mbox{for all } t>3.
  \eas
  Estimating
  \bas
	\int_3^t e^{-D\lam_1(t-s)} e^{-\beta_1 s} ds
	&=& e^{-D\lam_1 t} \int_3^t e^{(D\lam_1-\beta_1)s} ds \\
	&=& \frac{1}{D\lam_1-\beta_1} \big( e^{-\beta_1 t} - e^{-3\beta_1} e^{-D\lam_1(t-3)}\big) \\
	&\le& \frac{1}{D\lam_1-\beta_1} e^{-\beta_1 t}
	\qquad \mbox{for all } t>3,
  \eas
  from this we obtain (\ref{30.1}) if we let $\beta:=\beta_1$ and choose $C>0$ suitably large.
\qed
An interpolation similar to that from Lemma \ref{lem28} consequently yields the following.
\begin{lem}\label{lem31}
  Assume (\ref{g1}), (\ref{g2}) and (\ref{aL}). Then there exist $\beta>0$ and $C>0$ such that
  \be{31.1}
	\|\Theta_x(\cdot,t)\|_{L^\infty(\Om)} \le C e^{-\beta t}
	\qquad \mbox{for all } t>3.
  \ee
\end{lem}
\proof
  This can be deduced by an interpolation between the boundedness result from Lemma \ref{lem29} and the decay estimate
  obtained in Lemma \ref{lem30} on the basis of the Gagliardo-Nirenberg inequality stated in Lemma \ref{lem21},
  quite in the style of the reasoning in Lemma \ref{lem28}.
\qed
As $0\le t \mapsto \frac{1}{|\Om|} \io \Theta(\cdot,t)$ monotonically approaches a finite limit, Lemma \ref{lem31}
implies exponentially fast stabilization of $\Theta$ toward this spatially homogeneous limit.
\begin{lem}\label{lem32}
  If (\ref{g1}), (\ref{g2}) and (\ref{aL}) hold, then there exist $\Theta_\infty\ge 0$,
  $\beta>0$ and $C>0$ such that
  \be{32.1}
	\bigg| \io \Theta(\cdot,t) - \Theta_\infty \bigg| \le C e^{-\beta t}
	\qquad \mbox{for all } t>0.
  \ee
\end{lem}
\proof
  From (\ref{0}) and (\ref{v}) it follows that
  \bas
	\frac{d}{dt} \io \Theta = \io \gamma(\Theta) (v_x-au_x)^2
	\qquad \mbox{for all } t>0,
  \eas
  whence again using Lemma \ref{lem25}, Lemma \ref{lem23} and Lemma \ref{lem28} we infer that there exist $\beta>0$
  and $c_1>0$ such that
  \be{32.2}
	0 \le \frac{d}{dt} \io \Theta \le c_1 e^{-\beta t}
	\qquad \mbox{for all } t>2.
  \ee
  Therefore, with some $\Theta_\infty\ge 0$ we have
  \bas
	\io \Theta(\cdot,t)\nearrow \Theta_\infty
	\qquad \mbox{as } t\to\infty,
  \eas
  where once more by (\ref{32.2}),
  \bas
	\Theta_\infty - \io \Theta(\cdot,t)
	= \int_t^\infty \frac{d}{dt} \io \Theta(\cdot,s) ds
	\le c_1 \int_t^\infty e^{-\beta s} ds
	= \frac{c_1}{\beta} e^{-\beta t}
	\qquad \mbox{for all } t>2.
  \eas
  Due to the apparent boundedness of $\Theta\in \bom\times [0,2]$, this already yields the claim.
\qed
It remains to simply collect:\abs
\proofc of Theorem \ref{theo33}. \quad
  While (\ref{33.1}) and (\ref{33.2}) readily result from Lemma \ref{lem23}, Lemma \ref{lem28} and (\ref{mass}),
  the claim in (\ref{33.3}) is a consequence of Lemma \ref{lem31} when combined with Lemma \ref{lem32}.
\qed

\bigskip

\section*{Declarations}
{\bf Funding.} \quad
The authors acknowledge support of the Deutsche Forschungsgemeinschaft (Project No. 444955436).\abs
{\bf Conflict of interest statement.} \quad
The authors declare that they have no conflict of interest.\abs
and that they have no relevant financial or non-financial interests to disclose.\abs
{\bf Data availability statement.} \quad
Data sharing is not applicable to this article as no datasets were
generated or analyzed during the current study.

\end{document}